\documentclass[11pt]{amsart}
\usepackage{amssymb,amsmath,amsthm,amsfonts,amscd,euscript,hyperref,fourier,bm}
\usepackage[in]{fullpage}
\usepackage[backend=bibtex,
    sorting=anyt,
    isbn=false,
    url=false,
    doi=false,
    maxbibnames=100
    ]{biblatex}
\numberwithin{equation}{section}
\newtheorem{thm}{Theorem}[section]

\newtheorem{prop}[thm]{Proposition}
\newtheorem{lem}[thm]{Lemma}

\theoremstyle{remark}

\newcommand{\nc}{\newcommand}
\nc{\cO}{\mathcal O}
\nc{\cF}{\mathcal F}
\nc{\cL}{\mathcal L}
\nc{\msl}{\mathfrak{sl}}
\nc{\mgl}{\mathfrak{gl}}
\nc{\U}{\mathrm U}
\nc{\bH}{\EuScript H}
\nc{\Res}{\mathrm{Res\ }}
\nc{\Lie}{\mathrm{Lie\ }}

\newcommand{\bZ}{{\mathbb Z}}

\nc{\ch}{\mathrm{ch}}
\nc{\la}{\lambda}
\nc{\msp}{\mathfrak{sp}}
\nc{\cd}{\cdots}
\nc{\hk}{\hookrightarrow}
\nc{\T}{\otimes}
\nc{\al}{\alpha}
\nc{\om}{\omega}
\nc{\veps}{\varepsilon}
\nc{\ket}{\rangle}

\newcommand{\lk}{{\mathrm{linked}}}
\newcommand{\ulk}{{\mathrm{unlinked}}}
\newcommand{\gr}{{\mathrm{gr}}}

\begin{document}

\title{Categorification of quiver diagonalization and Koszul algebras}
\author{Vladimir Dotsenko}
\address{Institut de Recherche Math\'ematique Avanc\'ee, UMR 7501, Universit\'e de Strasbourg et CNRS, 7 rue Ren\'e-Descartes, 67000 Strasbourg, France}
\email{vdotsenko@unistra.fr}

\author{Evgeny Feigin}
\address{School of Mathematical Sciences, Tel Aviv University, Tel Aviv, 69978, Israel}
\email{evgfeig@gmail.com}

\author{Piotr Kucharski}
\address{Institute of Mathematics, University of Warsaw, ul. Banacha 2, 02-097 Warsaw, Poland}
\email{piotr.kucharski@mimuw.edu.pl}

\author{Markus Reineke}
\address{Ruhr-University Bochum, Faculty of Mathematics, Universit\"atstrasse 150, 44780 Bochum, Germany}
\email{Markus.Reineke@ruhr-uni-bochum.de}

\date{}

\begin{abstract} In earlier work of three of the authors of the present paper, a supercommutative quadratic algebra was associated to each symmetric quiver, and a new proof of positivity of motivic Donaldson--Thomas invariants of symmetric quivers was given using the so called numerical Koszul property of these algebras. It was furthermore conjectured that for each symmetric quiver such an algebra is Koszul. In this work, we lift the linking and unlinking operations on symmetric quivers of Ekholm, Longhi and the third author to the level of quadratic algebras, and use those lifts to prove the Koszulness conjecture.
\end{abstract}

\maketitle
\tableofcontents

\section{Introduction}
The motivic Donaldson-Thomas (DT) invariants of symmetric quivers, introduced in \cite{kontsevich_cohomological}, are natural analogues of Donaldson-Thomas invariants of sheaf counts. They are originally defined purely formally, by an Euler product factorization of the motivic generating series of stacks of representations of the quiver. Therefore, more structural interpretations are desirable, and several such have been developed, of geometric \cite{franzen_reineke,MR3034296,meinhardt_reineke}, algebraic \cite{MR4499100,DoMo,Efimov,kontsevich_cohomological}, or combinatorial \cite{MR4627326,MR2889742,MR4675069} nature.
The aim of the present work is to synthesize two of these approaches, enhancing the understanding on both the algebraic and the combinatorial side. Namely, we simultaneously consider the algebraic description of DT invariants in terms of (the Poincar\'e series of) Lie superalgebras of \cite{MR4499100}, and the combinatorial description of \cite{MR4627326}, reducing calculation of DT invariants to the case of multiple loop quivers, by iteration of linking and unlinking procedures modifying the quiver. 
These procedures are rooted in the knot-quiver correspondence \cite{MR3869892, MR4101645} and multiple cover generalizations of the skein relations for boundaries of holomorphic disks \cite{ekholm2021skeins}. More precisely, the application of such relations to disks corresponding to quiver nodes suggests that if linking or unlinking of their boundaries (hence the names of the operations) is accompanied by the appearance of an appropriate new node, then the motivic generating series is preserved, which was proved in \cite{MR4156213}. From the physical point of view, linking and unlinking corresponds to changing the representation of the BPS spectrum of the 3d $\mathcal{N}=2$ theory associated to the quiver, in which the modification of interactions is compensated by the appearance of a new BPS state \cite{MR4156213}.\\[1ex]
As our main result, we categorify these operations to the algebraic level, and use this categorification to prove the Koszulness conjecture of \cite{MR4499100}. More precisely, a quadratic algebra $\mathcal{A}_Q$ is associated in \cite{MR4499100} to every symmetric quiver $Q$, such that the above mentioned Lie superalgebra encoding the DT invariants of $Q$ arises via the Koszul dual of $\mathcal{A}_Q$; see Section 
\ref{sec:rec} for a detailed recollection. The algebra $\mathcal{A}_Q$ was conjectured to be Koszul in \cite{MR4499100} (although a weaker property of numerical Koszulness was sufficient to establish the relation to DT invariants), but standard techniques for proving Koszulness, such as quadratic Groebner bases, do not apply in general. In \cite{MR4627326}, to $Q$ and a choice of two vertices of $Q$, new quivers $Q^{\rm linked}$ and $Q^{\rm unlinked}$ are associated, whose motivic generating series are related to the one of $Q$ by a specialization of variables.\\[1ex]
These two approaches to DT invariants are combined in Sections \ref{sec:linking} and \ref{sec:unlinking}. In Theorem \ref{th:linking-gr}, it is proved that $\mathcal{A}_{Q^{\rm linked}}$ is isomorphic to the associated graded algebra of $\mathcal{A}_Q$ with respect to an appropriate filtration. In Theorem \ref{th:unlinking-h}, it is proved that, in turn, $\mathcal{A}_Q$ is isomorphic to the homology of $\mathcal{A}_{Q^{\rm unlinked}}$ with respect to an appropriate differential. The proof of that statement relies on the strategy of applying the linking procedure to the unlinked quiver, allowing one to combine the differential construction with the filtration construction into a very natural spectral sequence argument.  \\[1ex]
In Section \ref{sec:koszulness}, the key method of \cite{MR4627326} to produce a fully unlinked quiver (that is, a disjoint union of multiple loop quivers) by iterated linking/unlinking of $Q$ is recalled. Applying it to the level of the algebras $\mathcal{A}_Q$, this is shown to imply the Koszulness conjecture, reducing it to the known case of multiple loop quivers~\cite{MR4499100}.

\subsection*{Acknowledgements. } 
This work crucially benefited from the interaction the authors had during the workshop ``DT invariants of symmetric quivers: algebra, combinatorics, and topology'' organized thanks to financial support of the University of Strasbourg Institute for Advanced Study through the Fellowship USIAS-2021-061 within the French national program ``Investment for the future'' (IdEx-Unistra).

The work of V.D. was supported by Institut Universitaire de France and by the French national research agency project ANR-20-CE40-0016.
The work of P.K. was supported by the SONATA grant no. 2022/47/D/ST2/02058 funded by the Polish National Science Centre.
The work of M.R. was supported by the Deutsche Forschungsgemeinschaft grant CRC-TRR 191 “Symplectic structures in geometry, algebra and dynamics” (281071066).

\section{Recollections}\label{sec:rec}

All vector spaces and chain complexes in this article are defined over the field of rational numbers. All chain complexes are homologically graded, with the differential of degree $-1$. 

Throughout the article, $Q$ denotes a finite quiver with the set of vertices $Q_0$ and the set of arrows~$Q_1$. 
We shall assume that $Q$ is symmetric, that is, the number of arrows from $i$ to $j$ is equal to the number of arrows from $j$ to $i$ for all $i,j\in Q_0$. We denote by $M=(m_{ij})_{i,j\in Q_0}$ the incidence matrix of $Q$, and by $L$ the free abelian group $\mathbb{Z}^{Q_0}$, whose standard basis elements will be denoted by $\alpha_i$. The \emph{Euler form} of $Q$ is the bilinear form on $L$ defined as
$$\chi({\bf d},{\bf e})=\sum_{i\in Q_0}\mathbf{d}_i\mathbf{e}_i-\sum_{(a\colon i\rightarrow j) \in Q_1}\mathbf{d}_i\mathbf{e}_j.$$
Under our assumption, the Euler form is symmetric.

\subsection{Graded vector spaces and algebras}

Most vector spaces considered in this article are of the form
 \[
V=\bigoplus_{\mathbf{d}\in L}V_\mathbf{d}, \quad V_\mathbf{d}=\bigoplus_{(\mathbf{d},n)\in L\times\mathbb{Z}}V_\mathbf{d}^n .
 \]
We consider the category $\mathrm{Vect}^{L\times\mathbb{Z}}$ of such vector spaces, with morphisms being maps of degree zero. The category $\mathrm{Vect}^{L\times\mathbb{Z}}$ is monoidal, with the monoidal structures given by the tensor product $V\otimes W$ defined by
 \[
(V\otimes W)_\mathbf{d}^n=\bigoplus_{(\mathbf{d'},n')+(\mathbf{d''},n'')=(\mathbf{d},n)}V_\mathbf{d'}^{n'}\otimes W_{\mathbf{d''}}^{n''}.
 \]
Moreover, we can use the Koszul sign rule to define a braiding $\sigma\colon V\otimes W\to W\otimes V$ by the formula $\sigma(v\otimes w)=(-1)^{n'n''} w\otimes v$, for $v\otimes w\in V_\mathbf{d'}^{n'}\otimes W_{\mathbf{d''}}^{n''}$; this braiding makes the category $\mathrm{Vect}^{L\times\mathbb{Z}}$ symmetric monoidal. Finally, we shall use graded duals; for an object $V$, its graded dual $V^\vee$ is defined by the formula
 $$
V^\vee=\bigoplus_{\mathbf{d}\in L} V_\mathbf{d}^\vee,\qquad
V_\mathbf{d}^\vee=\bigoplus_{n\in\mathbb{Z}}(V_{\mathbf{d}}^{-n})^*.
 $$

We shall work with associative (super)commutative algebras in the category $\mathrm{Vect}^{L\times\mathbb{Z}}$. One can either define them directly, using the tensor product to talk about the structure operations and using the braiding to implement permutations of arguments, or, alternatively, one may note that the category $\mathrm{Vect}^{L\times\mathbb{Z}}$ contains the category $\mathrm{Vect}$ as a full symmetric monoidal subcategory of objects of degree zero, and so one may talk about objects in $\mathrm{Vect}^{L\times\mathbb{Z}}$ that are algebras over the operad $\mathsf{Com}$ in $\mathrm{Vect}$. 

\subsection{Poincar\'e series}\label{sec:PoSe}

Let $V\bigoplus_{k\in\bZ} V^k$ be a $\mathbb{Z}$-graded vector space with finite-dimensional components.  The \emph{Poincar\'e series} $P(V,q)$ is defined by the formula
 \[
P(V,q)=\sum_{k\in\bZ}\dim (V^k)(-q^{\frac12})^{-k}.
 \]
(Division by two corresponds to the standard convention used for cohomological Hall algebras, and using the exponent $-k$ comes from working with homologically graded vector spaces.) In general, this expression is an element of the vector space of doubly infinite Laurent series $\mathbb{Q}[[q^{\pm\frac12}]]$; in our work, we shall only deal with situations where it is finite on one of the sides, so that it belongs to one of the fields of formal Laurent series $\mathbb{Q}((q^{\frac12}))=\mathbb{Q}[[q^{\frac12}]][q^{-\frac12}]$ and $\mathbb{Q}((q^{-\frac12}))=\mathbb{Q}[[q^{-\frac12}]][q^{\frac12}]$. 

Let us consider formal variables $x_i$, $i\in Q_0$, and denote, for ${\bf d}\in L$, $x^{\bf d}=\prod_{i\in Q_0} x_i^{{\bf d}_i}$. 
To an object $V$ of $\mathrm{Vect}^{L\times\mathbb{Z}}$ with finite-dimensional components $V_\mathbf{d}^k$, we shall associate its \emph{Poincar\'e series} $P(V,x,q)$:
 \[
P(V,x,q)=\sum_{\mathbf{d}\in L} P(V_\mathbf{d},q)x^\mathbf{d}
=\sum_{\mathbf{d}\in L}\sum_{k\in\mathbb{Z}}
\dim (V_\mathbf{d}^k)(-q^{\frac12})^{-k}x^\mathbf{d}.
 \]
For certain objects in $\mathrm{Vect}^{L\times\mathbb{Z}}$, we shall use the fact that Poincar\'e series behave well with respect to various operations, including tensor products. For that, it is important to consider a smaller category. Our choice is to consider the subcategory $\mathcal{C}$ consisting of objects 
$$
V=\bigoplus_{\mathbf{d}\in\mathbb{Z}_{\ge0}^{Q_0}}V_\mathbf{d}
=\bigoplus_{\mathbf{d}\in\mathbb{Z}_{\ge0}^{Q_0}}\bigoplus_{k\in\mathbb{Z}}V_\mathbf{d}^k$$
such that all components $V_\mathbf{d}^k$ are finite-dimensional and $V_\mathbf{d}^k=0$ for $k\gg0$.
The Poincar\'e series is a ring homomorphism 
 \[
K_0(\mathcal{C})\to R_Q=\mathbb{Q}((q^{\frac12}))[[x_i\colon i\in Q_0]].
 \]
We emphasize that not all objects we consider belong to $\mathcal{C}$, but whenever we use the compatibility of Poincar\'e series with tensor products of vector spaces, we shall restrict ourselves to objects from $\mathcal{C}$.

For a symmetric quiver $Q$, its \emph{motivic generating series} $A_Q(x,q)$ is defined by the formula 
$$A_Q(x,q)=\sum_{{\bf d}\in\mathbb{Z}_{\ge 0}^{Q_0}}\frac{(-q^{\frac12})^{-\chi({\bf d},{\bf d})}x^{\bf d}}{\prod_{i\in Q_0} (q^{-1})_{\mathbf{d}_i}} \in R_Q,$$
where $(q)_n=(1-q)\cdot\dots\cdot(1-q^n)$. This generating series is the Poincar\'e series of the multigraded vector space
 \[
\mathcal{H}_Q=\bigoplus_{\bf d} H^{*+\chi({\bf d},{\bf d})}_{G_{\bf d}}(R_{\bf d}(Q),\mathbb{Q}),
 \]
where the individual graded pieces are spaces of multisymmetric polynomials:
 \[
H^*_{G_{\bf d}}(R_{\bf d}(Q),\mathbb{Q})=H^*_{G_{\bf d}}(\mathrm{pt})\cong \Lambda_\mathbf{d}
 \]
with 
 \[
\Lambda_\mathbf{d}=\mathbb{Q}[z_{i,r}\colon i\in Q_0,\, 1\le r\le \mathbf{d}_i]^{\Sigma_\mathbf{d}},\qquad
\Sigma_\mathbf{d}=\prod_{i\in Q_0}\Sigma_{\mathbf{d}_i}.
 \]

\subsection{Quadratic algebras associated to symmetric quivers}\label{sec:quad-alg}

In \cite{MR4499100}, to each symmetric quiver $Q$, a (super)commutative associative algebra $\mathcal{A}_Q$ was associated. It has generators $e_{i,k}$ with $i\in Q_0$, $k\ge 0$; we set $\deg(e_{i,k})=(\alpha_i,-2k-m_{i,i})\in L\times\mathbb{Z}$, so that the vector space spanned by the generators is an object of $\mathcal{C}$. To express its relations, we shall use the formal generating series 
\[
e_i(z)=\sum_{k\ge 0} e_{i,k}z^k.
\] 
The first group of relations is obtained by extracting individual coefficients of 
 \[
e_i(z)e_j(w) - (-1)^{m_{i,i}m_{j,j}}e_j(w)e_i(z) = 0, 
 \]
these relations simply mean that the algebra $\mathcal{A}_Q$ is indeed supercommutative. The second group of relations states that 
 \[
e_i(z)\frac{d^p\phantom{z}}{dz^p}e_j(z)=0 \text{ for all } i,j \text{ and all } 0\le p<m_{i,j}.
 \]
It is easy to show that this group of relations is equivalent to a larger system of relations  
 \[
\frac{d^p\phantom{z}}{dz^p}e_i(z)\frac{d^q\phantom{z}}{dz^q}e_j(z)=0 \text{ for all } i,j \text{ and for } 0\le p+q<m_{i,j}.
 \]

Let us now recall, following \cite[Prop.~3.4]{MR4499100}, a convenient way to interpret the graded dual space $\mathcal{A}_Q^\vee$ of the algebra $\mathcal{A}_Q$. Suppose that $\xi\in \mathcal{A}_{Q,\mathbf{d}}^\vee$. Let us consider formal variables $z_{i,r}$, $i\in Q_0, 1\le r\le \mathbf{d}_i$, of homological degree $2$, and for each such variable, we form the corresponding series $e_i(z_{i,r})$. Note that under our conventions each term $e_{i,k}z_{i,r}^k$ is of homological degree $-m_{i,i}$, so the expression 
 \[
\prod_{i\in Q_0} e_i(z_{i,1})\dots e_i(z_{i,\mathbf{d}_i})
 \]
is of homological degree $-\mathbf{m}\cdot\mathbf{d}:=-\sum_{i\in Q_0} m_{i,i}\mathbf{d}_i$, and thus the evaluation 
 \[
f_\xi=\xi \Bigl(\prod_{i\in Q_0} e_i(z_{i,1})\dots e_i(z_{i,\mathbf{d}_i})\Bigr)
 \]
is a map of graded vector spaces from $\mathcal{A}_{Q,\mathbf{d}}^\vee$ to the degree shifted polynomial ring 
 \[
\mathbb{Q}[z_{i,r}\colon  i\in Q_0,\, 1\le r\le \mathbf{d}_i][-\mathbf{m}\cdot\mathbf{d}]. 
 \] 
In fact, it is possible to prove that the image of that space is precisely all multisymmetric polynomials divisible by
\[
F_\mathbf{d}:=\prod_{i\in Q_0}\prod_{1\le r<r'\le \mathbf{d}_i}(z_{i,r}-z_{i,r'})^{m_{i,i}} \prod_{\substack{i,i'\in Q_0, i\ne i'\\ 1\le r\le \mathbf{d}_i,\\1\le r'\le \mathbf{d}_{i'}}}(z_{i,r}-z_{i',r'})^{m_{i,i'}}.
 \]
We call $F_\mathbf{d}\Lambda_{\mathbf{d}}[-\mathbf{m}\cdot\mathbf{d}]$ the \emph{functional realizatiion} of the graded dual space $\mathcal{A}_Q^\vee$. This realization implies the following result, connecting these algebras to motivic generating series.

\begin{prop}[{\cite[Prop.~3.6]{MR4499100}}]\label{prop:poincare series}
In the ring $\mathbb{Q}(q^{\frac12})[[x_i\colon i\in Q_0]]$, we have 
 \[
A_Q(x,q)=P(\mathcal{A}_Q,q^{\frac12}x,q).
 \]
\end{prop}

\subsection{Koszul algebras}\label{sec:koszulalg}

We refer the reader to the monograph \cite{PP} for detailed discussion of Koszul duality for associative algebras. For our purposes, the following viewpoint will be sufficient. Let $A$ be a (super)commutative associative algebra which we shall assume to have a weight grading, and, moreover, to be generated by elements of weight grading one. One can construct a multiplicative free resolution of this algebra (a differential graded algebra whose homology is isomorphic to $A$) by applying the construction of ``killing cycles'' of Tate \cite{MR268172,MR86072} and adjoining, in order, generators of homological degree one whose differentials are equal to relations of $A$, then generators of homological degree two whose differentials are equal to basis elements of the module of $1$-cycles, etc. It is common to refer to homological degree in this context as \emph{syzygy degree} (see e.g. \cite{LV}), which is particularly useful in the context like ours where elements of the original algebra have their own internal homological degree that we would not want to mix with the degree of newly added generators. 

If one performs the Tate construction in the most ``economic'' way, it produces the \emph{minimal model} of the algebra $A$, and our algebra is said to be Koszul if the differential is \emph{quadratic}, that is sends every generator into a linear combination of products of pairs of generators. (This implies, inductively, that homological degree of each generator of the thus obtained model is one less than the weight grading of that generator, leading to an equivalent definition of a Koszul algebra \cite{PP}.)  

In \cite{MR4499100}, it was proved that the Euler characteristics of the multigraded components of the minimal model of the algebra $\mathcal{A}_Q$  for every symmetric quiver $Q$ satisfy some relations that would have been satisfied if that algebra were Koszul (one says that $\mathcal{A}_Q$ is ``numerically Koszul''). Moreover, it was established in the same paper that the algebra $\mathcal{A}_Q$ is Koszul if each connected component of $Q$ is either the doubling of the $A_2$ quiver, or is very ``regular'' in the following sense: there exists a positive integer $N$ such that for all vertices $i\ne j$ of that component, we have $m_{ij}=N$, and for each vertex $i$ of that component we have either $m_{ii}=N$ or $m_{ii}=N+1$. It was conjectured that the algebra $\mathcal{A}_Q$ is Koszul for every symmetric quiver $Q$; in this paper, we prove that conjecture. 

\section{Linking and filtrations}\label{sec:linking}

\subsection{Linking}

Let us recall the linking procedure for symmetric quivers introduced in \cite{MR4089349}. Suppose that $Q$ is a symmetric quiver and $a\ne b$ are two elements of $Q_0$. We define the linked quiver $Q^\lk$ as the quiver with the incidence matrix $M^\lk=\left(m^\lk_{i,j}\right)_{i,j\in Q_0\sqcup\{\diamond\}}$, the symmetric matrix with
 \[
m^\lk_{i,j}=
\begin{cases}
\quad\quad m_{i,j}+1,\qquad\qquad\quad  \! \{i,j\}=\{a,b\},\\
%\quad m_{a,a}+m_{a,b},\qquad\qquad\,  \{i,j\}=\{a,\diamond\},\\
%\quad m_{a,b}+m_{b,b},\qquad\qquad\,\, \{i,j\}=\{b,\diamond\},\\
\quad m_{i,a}+m_{i,b},\qquad\qquad\quad\!\!  i\in Q_0\setminus\{a,b\}, j=\diamond,\\
\quad m_{j,a}+m_{j,b},\qquad\qquad\quad\!\! i=\diamond ,\in Q_0\setminus\{a,b\},\\
%\quad m_{i,a}+m_{i,b},\qquad\qquad\,  j=\diamond,\\
m_{a,a}+m_{b,b}+2m_{a,b},\qquad\!  i=j=\diamond,\\
\quad\qquad m_{i,j},\qquad\qquad\qquad\! \text{ for all other } i,j\in Q_0.
\end{cases}
 \] 

This construction also agrees with motivic generating series in a very transparent way. 

\begin{thm}[{\cite[Sec.~4.5]{MR4089349}}]\label{thm:motivic-linking}
We have the following equality of motivic generating series.
 \[
A_Q(x,q)=A_{Q^\lk}(x,q)|_{x_{\diamond}=x_ax_b}.
 \]
\end{thm}

\subsection{The filtration theorem}\label{sec:filtr-thm}

We shall now explain an algebraic construction that categorifies linking to a certain extent. Let $a\ne b\in Q_0$ be the vertices used in the linking procedure above. We shall consider the formal generating series 
 \[
e_{\diamond}(z)=e_a(z)\frac{d^{m_{a,b}}\phantom{z}}{dz^{m_{a,b}}}e_b(z) 
 \]
and use it to introduce a filtration $F_\bullet$ on the algebra $\mathcal{A}_Q$ by putting all generators $e_{i,k}$ and all coefficients $e_{\diamond,k}$ of $e_{\diamond}(z)$ in the first filtration component $F_1\mathcal{A}_Q$, and defining the component $F_p\mathcal{A}_Q$ as the vector space spanned by products of at most $p$ factors from $F_1\mathcal{A}_Q$. This is an exhaustive increasing filtration of $\mathcal{A}_Q$ which is clearly compatible with the product: $F_{p_1}\mathcal{A}_QF_{p_2}\mathcal{A}_Q\subset F_{p_1+p_2}\mathcal{A}_Q$. Thus, the associated graded object
 \[
\gr_{F}\mathcal{A}_Q = \bigoplus_{p\ge 0} F_p\mathcal{A}_Q/F_{p-1}\mathcal{A}_Q
 \]
has a well defined algebra structure equipped with the extra weight grading arising from the filtration, and that algebra is generated by (the cosets of) the original generators $e_{i,k}$ and all elements $e_{\diamond,k}$, which all have the new weight grading equal to one.  

\begin{prop}
The following relations are satisfied in $\gr_{F}\mathcal{A}_Q$:
 \[
e_i(z)\frac{d^p\phantom{z}}{dz^p}e_j(z)=0\quad
\begin{cases}
\text{ for } i=a, j=b, \text{ and } 0\le p < m_{a,b}+1,\\
%\text{ for } i=a, j=\diamond, \text{ and } 0\le p < m_{a,a}+m_{a,b},\\
%\text{ for } i=b, j=\diamond,  \text{ and }0\le p < m_{a,b}+m_{b,b},\\
%\text{ for } i\in \{a,b\}, j=\diamond,  \text{ and }0\le p < m_{a,b}+m_{b,b},\\
\text{ for } i\in Q_0, j=\diamond,\text{ and } 0\le p < m_{i,a}+m_{i,b},\\
\text{ for }  i=j=\diamond, \text{ and } 0\le p < m_{a,a}+m_{b,b}+2m_{a,b},\\
\text{ for all other } i,j\in Q_0 \text{ and } 0\le p < m_{i,j}.
\end{cases}
 \]

\end{prop}
\begin{proof}
The first group of relations is satisfied since we defined
\[
e_{\diamond}(z)=e_a(z)\frac{d^{m_{a,b}}\phantom{z}}{dz^{m_{a,b}}}e_b(z), 
 \]
thus the coefficients of the series $e_a(z)\frac{d^{m_{a,b}}\phantom{z}}{dz^{m_{a,b}}}e_b(z)$ (elements of $F_2\mathcal{A}_Q$), being equal to $e_{\diamond,k}$, vanish in the quotient $F_2\mathcal{A}_Q/F_1\mathcal{A}_Q$.

All the remaining relations already hold on the nose in $\mathcal{A}_Q$, which one can easily see for each of them using the functional realization of the graded dual space. As an example, let us prove that the relations
 \[
e_\diamond(z)\frac{d^p\phantom{z}}{dz^p}e_\diamond(z)=0, \qquad 0\le p < m_{a,a}+m_{b,b}+2m_{a,b}
 \]
are satisfied. In algebra $\mathcal{A}_Q$, let us consider the subspace spanned by the coefficients of the formal power series 
 \[
e_a(z_1)e_a(z_2)e_b(w_1)e_b(w_2).
 \]  
In order to show that $e_\diamond(z)\frac{d^p\phantom{z}}{dz^p}e_\diamond(z)$ vanishes for some $p$, 
we should show that all linear functions from the graded dual vector space vanish on that element. 

We know that the functional realization assigns to a linear function $\xi$ the element
\[
\xi(e_a(z_1)e_a(z_2)e_b(w_1)e_b(w_2))=\pm\xi(e_a(z_1)e_b(w_1)e_a(z_2)e_b(w_2)),
\]
so evaluating $\xi$ on $e_\diamond(z)\frac{d^p\phantom{z}}{dz^p}e_\diamond(z)$ amounts, up to a sign, to first computing   
 \[
\left(\frac{d^{m_{a,b}}\phantom{z}}{dw_1^{m_{a,b}}}\frac{d^{m_{a,b}}\phantom{z}}{dw_2^{m_{a,b}}}\xi(e_a(z_1)e_a(z_2)e_b(w_1)e_b(w_2))\right)_{w_1=z_1,w_2=z_2},
 \]
and then examining the order of vanishing of that element for $z_1=z_2$.
Since we know that the functional realization of corresponding component of the graded dual vector (for the degree $2\alpha_a+2\alpha_b$) consists of multi-symmetric polynomials divisible by
 \[
(z_1-z_2)^{m_{a,a}}(w_1-w_2)^{m_{b,b}}
(z_1-w_1)^{m_{a,b}}(z_1-w_2)^{m_{a,b}}(z_2-w_1)^{m_{a,b}}
(z_2-w_2)^{m_{a,b}},
 \]
applying $\frac{d^{m_{a,b}}\phantom{z}}{dw_1^{m_{a,b}}}\frac{d^{m_{a,b}}\phantom{z}}{dw_2^{m_{a,b}}}$ and setting $w_1=z_1$, $w_2=z_2$ may give a non-zero result only if we apply all the derivatives to the product $(z_1-w_1)^{m_{a,b}}(z_2-w_2)^{m_{a,b}}$. Then this product disappears, and after setting $w_1=z_1$, $w_2=z_2$ we obtain a factor $(z_1-z_2)^{m_{a,a}+m_{b,b}+2m_{a,b}}$, as desired.
\end{proof}

Since the algebra $\gr \mathcal{A}_Q$ is generated by (the cosets of) the original generators $e_{i,k}$ and all elements $e_{\diamond,k}$, the result we just proved implies that there is a surjective algebra morphism 
 \[
\pi_\lk\colon \mathcal{A}_{Q^\lk}\twoheadrightarrow \gr \mathcal{A}_Q.
 \]
We are now ready to state and prove the main result of this section.

\begin{thm}\label{th:linking-gr}
The algebra morphism $\pi_\lk$ implements an isomorphism
 \[
\mathcal{A}_{Q^\lk}\cong \gr \mathcal{A}_Q.
 \]
\end{thm}

\begin{proof}
Let us note that since we defined
 \[
e_{\diamond}(z)=e_a(z)\frac{d^{m_{a,b}}\phantom{z}}{dz^{m_{a,b}}}e_b(z),  
 \]
in algebra $\mathcal{A}_Q$ we have 
 \[
\deg(e_{\diamond,k})=(\alpha_a+\alpha_b,-2(k+m_{a,b})-m_{a,a}-m_{b,b})=(\alpha_a+\alpha_b,-2k-m^\lk_{\diamond,\diamond}). 
 \]
Thus, if we want the algebra morphism $\pi_\lk$ to respect the gradings, we cannot use the full $L^\lk$-degree on the algebra 
$\mathcal{A}_{Q^\lk}$ but rather reduce it to $L$-degree, postulating that the $L$-degree of the components of $e_\diamond(z)$ is $\alpha_a+\alpha_b$. On the level of the Poincar\'e series, this amounts to the substitution $x_{\diamond}=x_ax_b$.

Let us now remark that according to Theorem \ref{thm:motivic-linking} and Proposition \ref{prop:poincare series}, we have
 \[
P(\mathcal{A}_Q,q^{\frac12}x,q)=A_Q(x,q)=A_{Q^\lk}(x,q)|_{x_{\diamond}=x_ax_b}=P(\mathcal{A}_{Q^\lk},q^{\frac12}x,q)|_{x_{\diamond}=x_ax_b},
 \]
implying that we have the equality of dimensions of all respective $L\times\mathbb{Z}$-graded components of the algebras $\mathcal{A}_{Q^\lk}$ and $\mathcal{A}_Q$. For each value of $L\times\mathbb{Z}$-grading, the map $\pi_\lk$ is thus a surjective map of finite-dimensional vector spaces of the same dimension, which therefore has to be an isomorphism. 
\end{proof}

\section{Unlinking and partial resolutions}\label{sec:unlinking}

\subsection{Unlinking}

Let us recall the unlinking procedure for symmetric quivers introduced in \cite{MR4089349}. Suppose that $Q$ is a symmetric quiver and $a\ne b$ are two elements of $Q_0$ such that $m_{a,b}>0$. We define the unlinked quiver $Q^\ulk$ as the quiver with the incidence matrix $M^\ulk=\left(m^\ulk_{i,j}\right)_{i,j\in Q_0\sqcup\{\star\}}$, the symmetric matrix with
 \[
m^\ulk_{i,j}=\quad
\begin{cases}
\quad\quad m_{i,j}-1,\qquad\qquad\qquad\,\, \{i,j\}=\{a,b\},\\
\quad m_{a,a}+m_{a,b}-1,\qquad\qquad  \{i,j\}=\{a,\star\},\\
\quad m_{b,b}+m_{a,b}-1,\qquad\qquad\, \{i,j\}=\{b,\star\},\\
%\quad m_{i,a}+m_{i,b}-1,\qquad\qquad  i\in \{a,b\},j=\star,\\
\quad m_{i,a}+m_{i,b},\qquad\qquad\qquad\!  i\in Q_0\setminus\{a,b\}, j=\star,\\
\quad m_{j,a}+m_{j,b},\qquad\qquad\qquad  i=\star, j\in Q_0\setminus\{a,b\},\\
m_{a,a}+m_{b,b}+2m_{a,b}-1,\quad\,\,  i=j=\star,\\
\quad\qquad m_{i,j},\qquad\qquad\qquad\quad \text{ for all other } i,j\in Q_0.
\end{cases}
 \] 

The importance of this construction is explained by the following result (note that we use a different convention for the variable $q$, hence a slight discrepancy with the quoted result of \cite{MR4089349}). 

\begin{thm}[{\cite[Sec.~4.2]{MR4089349}}]\label{thm:motivic-unlinking}
We have the following equality of motivic generating series.
 \[
A_Q(x,q)=A_{Q^\ulk}(x,q)|_{x_{\star}=q^{-\frac12}x_ax_b}.
 \]
\end{thm}

\subsection{The partial resolution theorem}\label{sec:partresol}

We shall now explain an algebraic construction that categorifies unlinking to a certain extent. Let $a\ne b\in Q_0$ be the vertices used in the unlinking procedure above, where we assume $m_{a,b}>0$. This means that one of the relations of the algebra $\mathcal{A}_Q$ is
 \[
e_a(z)\frac{d^{m_{a,b}-1}\phantom{z}}{dz^{m_{a,b}-1}}e_b(z)=0.
 \]
Let us use the quiver $Q^\ulk$ to construct the algebra $\mathcal{A}_{Q^\ulk}$ according to our general recipe. In that algebra, the formal series
 \[
e_a(z)\frac{d^{m_{a,b}-1}\phantom{z}}{dz^{m_{a,b}-1}}e_b(z)
 \]
does not vanish, and its coefficient of $z^k$ has the multidegree 
 \[
(\alpha_a+\alpha_b,-2(k+m_{a,b}-1)-m_{a,a}-m_{b,b})=(\alpha_a+\alpha_b,-2k-(m_{\star,\star}-1)). 
 \]

\begin{prop} \label{prop:deriv}
There is a unique odd derivation $\partial\colon \mathcal{A}_{Q^\ulk}\to \mathcal{A}_{Q^\ulk}$ for which
 \[
\partial(e_{\star}(z))= e_a(z)\frac{d^{m_{a,b}-1}\phantom{z}}{dz^{m_{a,b}-1}}e_b(z), \quad \partial(e_i(z))=0, i\in Q_0
 \]
component-wise. Moreover, we have $\partial^2=0$.
\end{prop}

\begin{proof}
Since we defined $\partial$ on all generators, the derivation property immediately implies that there is at most one way to extend it to derivation. To show that it extends to a well defined derivation, we need to check that it preserves the vector space spanned by quadratic relations. The relations of the algebra $\mathcal{A}_{Q^\ulk}$ that do not involve $e_{\star}(z)$ do not pose a problem, since $\partial$ vanishes on the corresponding generators. All relations of the first kind (supercommutativity relations) do not pose a problem either, since derivations are compatible with (super)commutators. Compatibility with all the remaining relations is checked by a simple direct calculation.  As an example, for $i\ne\{a,b\}$, the conditions stating that the element 
$e_\star(z)e_i(w)$ vanishes at $z=w$ to order $m_{i,\star}=m_{i,a}+m_{i,b}$ are sent by $\partial$ to the conditions stating that 
 \[
e_a(z)\frac{d^{m_{a,b}-1}\phantom{z}}{dz^{m_{a,b}-1}}(e_b(z))e_i(w)
 \]
vanishes at $z=w$ to order $m_{i,\star}=m_{i,a}+m_{i,b}$, which is true since the latter is the first nonzero coefficient of the expansion of 
 \[
e_a(z)e_b(z')e_i(w)
 \]
around $z'=z$, and this latter expression vanishes to order $m_{i,a}$ at $w=z$ and to order $m_{i,b}$ at $w=z'$. Similar arguments apply to all other relations of the second type. 

To prove the second assertion, we remark that since the derivation $\partial$ is odd, we have that $\partial^2=\frac12[\partial,\partial]$ is a derivation, so the fact that $\partial^2$ clearly vanishes on all generators implies that $\partial^2=0$. 
\end{proof}

Since $\partial^2=0$, and $\partial$ is a derivation, the cohomology 
 \[
H_\bullet(\mathcal{A}_{Q^\ulk},\partial)
 \]
has an induced algebra structure. We are now ready to state and prove the main result of this section.

\begin{thm}\label{th:unlinking-h}
We have an algebra isomorphism
 \[
H_\bullet(\mathcal{A}_{Q^\ulk},\partial)\cong \mathcal{A}_Q.
 \]
\end{thm}

\begin{proof}
Let us consider the linking procedure of the vertices $a$ and $b$ of the quiver $Q^\ulk$. As we established in Theorem \ref{th:linking-gr}, the associated graded of the algebra $\mathcal{A}_{Q^\ulk}$ with respect to the filtration defined in Section \ref{sec:filtr-thm} is precisely the algebra $\mathcal{A}_{(Q^\ulk)^\lk}$. The incidence matrix of the quiver $(Q^\ulk)^\lk$ is given by 
 \[
(m^\ulk)^\lk_{i,j}=
\begin{cases}
\quad\qquad m_{i,j},\qquad\qquad\qquad\quad \text{ for all } i,j\in Q_0,\\
%\quad m_{a,a}+m_{a,b}-1,\qquad\qquad  i=a, j\in\{\star,\diamond\}\\
%\quad m_{b,b}+m_{a,b}-1,\qquad\qquad\,  i=b, j\in\{\star,\diamond\}\\
\quad m_{i,a}+m_{i,b}-1,\qquad\qquad\,  i\in\{a,b\}, j\in\{\star,\diamond\},\\
\quad m_{j,a}+m_{j,b}-1,\qquad\qquad\,  i\in\{\star,\diamond\}, j\in\{a,b\}, \\
\quad m_{i,a}+m_{i,b},\qquad\qquad\qquad  i\in Q_0\setminus\{a,b\}, j\in\{\star,\diamond\},\\
\quad m_{j,a}+m_{j,b},\qquad\qquad\qquad  i\in\{\star,\diamond\}, j\in Q_0\setminus\{a,b\}, \\
m_{a,a}+m_{b,b}+2m_{a,b}-1,\quad\,  i=j=\star,\\
m_{a,a}+m_{b,b}+2m_{a,b}-2,\qquad\!\!\!  \{i,j\}=\{\star,\diamond\},\\
m_{a,a}+m_{b,b}+2m_{a,b}-2,\qquad\!\!\!  i=j=\diamond.
\end{cases}
 \] 
Moreover, the map $\partial$ induces an odd derivation $\gr_F\partial$ of the algebra $\mathcal{A}_{(Q^\ulk)^\lk}$ which is of particularly simple form: we have
 \[
\gr_F\partial(e_\star(z))=e_\diamond(z),
 \]
and all other generators are annihilated by $\gr_F\partial$. Let us show that  
  \[
H_\bullet(\mathcal{A}_{(Q^\ulk)^\lk},\gr_F\partial)\cong \mathcal{A}_Q.
 \]
To that end, we shall now define a contracting homotopy for $\gr_F\partial$. That is done using the following lemma.

\begin{lem} 
There is a unique odd derivation $h\colon \mathcal{A}_{(Q^\ulk)^\lk}\to \mathcal{A}_{(Q^\ulk)^\lk}$ which satisfies
 \[
h(e_{\diamond}(z))= e_\star(z) 
 \]
component-wise, and annihilates all other generators.  
\end{lem} 

\begin{proof}
As in Proposition \ref{prop:deriv}, we need to show that this derivation is compatible with relations. A relation that does not involve $e_{\diamond}(z)$ is annihilated by $h$. All relations of the first kind (supercommutativity relations) do not pose a problem either, since derivations are compatible with (super)commutators. Let us consider the remaining cases. Since for $i\in Q_0$ we have $(m^\ulk)^\lk_{i,\star}=(m^\ulk)^\lk_{i,\diamond}$, the corresponding relations are preserved. The relation stating that $e_\diamond(z)e_\star(w)$ vanishes at $z=w$ to order $m_{\diamond, \star}=m_{a,a}+m_{b,b}+2m_{a,b}-2$ is sent to the relation stating that $e_\star(z)e_\star(w)$ vanishes at $z=w$ to the same order, and that is true even to order one higher. Finally, the relation stating that $e_\diamond(z)e_\diamond(w)$ vanishes at $z=w$ to order $m_{\diamond, \diamond}=m_{a,a}+m_{b,b}+2m_{a,b}-2$ is sent to the relation stating that $e_\diamond(z)e_\star(w)\pm e_\star(z)e_\diamond(w)$ vanishes at $z=w$ to the same order, and that is true for each of the two terms individually. 
\end{proof}

We now remark that the commutator $[\gr_F\partial,h]$ is a derivation which satisfies
 \[
[\gr_F\partial,h](e_\star(z))=e_\star(z), \quad [\gr_F\partial,h](e_\diamond(z))=e_\diamond(z),
 \]
and annihilates all other generators. Thus, on each homogeneous component $(\mathcal{A}_{(Q^\ulk)^\lk})_{\mathbf{d}}$ this derivation is the multiplication by $\mathbf{d}_\star+\mathbf{d}_\diamond$. Note that the cochain complex
$(\mathcal{A}_{(Q^\ulk)^\lk},\gr_F\partial)$
decomposes into a direct sum of subcomplexes consisting of all $L$-homogeneous components with the given fixed value of $\mathbf{d}_\star+\mathbf{d}_\diamond$. As we just established, each such complex is contractible unless $\mathbf{d}_\star+\mathbf{d}_\diamond=0$. The latter complex is clearly just $\mathcal{A}_{Q}$ with zero differential. 

We therefore established that the homology of the associated graded object is $\mathcal{A}_Q$. Let us explain that this implies the assertion of the theorem. Indeed, the homology of the chain complex $(\mathcal{A}_{Q^\ulk},\partial)$ in the syzygy degree zero is clearly isomorphic to $\mathcal{A}_Q$ as an algebra, since the differential $\partial$ is designed specifically for that. On the other hand, we have just shown that  
 \[
H_\bullet(\mathcal{A}_{(Q^\ulk)^\lk},\gr_F\partial)\cong \mathcal{A}_Q,
 \]
and normally one has a spectral sequence with the first page $H^\bullet(\mathcal{A}_{(Q^\ulk)^\lk},\gr_F\partial)$ that converges to $H^\bullet(\mathcal{A}_{Q^\ulk},\partial)$, but since in our case the first page is concentrated in single syzygy degree, there is no room for higher differentials.
\end{proof}

\section{The Koszulness conjecture}\label{sec:koszulness}

We are now ready to prove the Koszulness conjecture of \cite{MR4499100}.

\begin{thm}
For each symmetric quiver $Q$, the algebra $\mathcal{A}_Q$ is Koszul.
\end{thm}

\begin{proof}
To begin, we briefly recall the quiver diagonalization procedure of \cite{MR4627326}. That procedure constructs, for a symmetric quiver $Q$ 
(with the vertex set $Q_0$) and for each $n\ge 1$, 
\begin{itemize}
\item[-] a finite set $V^{(n)}$ and a symmetric quiver $Q^{(n)}$ with the set of vertices 
 \[
Q^{(n)}_0=Q_0\sqcup\bigsqcup_{1\le k\le n}V^{(k)},
 \] 
\item[-] monomials $m_i^{(n)}$%\in\mathbb{N}^{Q_0}$ 
in variables $x_j$, $j\in Q_0$, and $q^{\pm1}$, which are indexed by $i\in V^{(n)}$, 
\end{itemize}
such that  
\begin{itemize}
\item[-] each quiver $Q^{(n)}$ is diagonal, that is only has loops among its arrows, 
\item[-] we have 
 \[
A_Q\left(\left\{x_i\colon i\in Q_0\right\},q\right)-\left.A_{Q^{(n)}}\left(\left\{\tilde x_i\colon i\in Q^{(n)}_0\right\},q\right)\right|_{\tilde x_j:= m_j^{(k)},\, j\in V^{(k)}, 1\le k\le n}=O(x^{n+1}).
 \]
\end{itemize}
Essentially, one first constructs the sequence of quivers $\bar{Q}^{(n)}$ by putting $\bar{Q}^{(0)}:=Q$, and stating that for each $k\ge 1$, the quiver $\bar{Q}^{(k)}$ is obtained from $\bar{Q}^{(k-1)}$ by applying unlinking to remove all the edges between different vertices of $\bar{Q}^{(k-1)}_0$; by definition, we denote by $V^{(k)}:=\bar{Q}^{(k)}_0\setminus \bar{Q}^{(k-1)}_0$ the set of newly obtained vertices. The quiver $Q^{(n)}$ is obtained from $\bar{Q}^{(n)}$ by forgetting all edges that are not loops. The novel idea of \cite{MR4627326} was to consider the limiting quiver $Q^{(\infty)}=\lim_{n\to\infty} Q^{(n)}$, whose motivic generating series is on the one hand equal to the motivic generating series of $Q$, and on the other hand factorizes as the product of series obtained from motivic generating series of one-vertex quiver by substitutions of various monomials instead of the variable $x$, thus expressing motivic DT invariants of $Q$ via motivic DT invariants of one-vertex quivers. We shall now explain how this can be categorified into a proof of Koszulness of the algebra $\mathcal{A}_Q$.

Let us consider the quiver $\bar{Q}^{(n)}$ for some $n$. It follows from the results of Section \ref{sec:partresol} that the algebra~$\mathcal{A}_{\bar{Q}^{(n)}}$ can be equipped with a differential $\partial^{(n)}$ such that 
 \[
H_\bullet(\mathcal{A}_{\bar{Q}^{(n)}},\partial^{(n)})\cong \mathcal{A}_Q.
 \] 
Indeed, it is enough to put $\partial^{(n)}$ to be the sum of all the differentials corresponding to simple unlinkings that we constructed in Proposition \ref{prop:deriv}. It is clear that such differentials corresponding to unlinking of two different pairs of vertices of $\bar{Q}^{(k-1)}_0$ pairwise anticommute, so their sum is a differential. One may compute homology of that differential inductively by iterating the spectral sequence of a bicomplex; at each step the result of Theorem \ref{th:unlinking-h} ensures that the homology is isomorphic to $\mathcal{A}_Q$.  

Considering the limit for $n\to\infty$, we obtain a differential $\partial^{(\infty)}$ on the algebra $\mathcal{A}_{Q^{(\infty)}}$ such that
 \[
H_\bullet(\mathcal{A}_{Q^{(\infty)}},\partial^{(\infty)})\cong \mathcal{A}_Q.
 \] 
Moreover, this differential is quadratic. Since the quiver $Q^{(\infty)}$ is diagonal, we have an isomorphism of algebras
 \[
\mathcal{A}_{Q^{(\infty)}}\cong \bigotimes_{i\in Q^{(\infty)}_0} \mathcal{A}_i,
 \]
where each algebra $\mathcal{A}_i$ corresponds to a certain one-vertex quiver. It is established in \cite{MR4499100} that such algebras are Koszul, so they admit free commutative algebra resolutions $(\mathcal{B}_i,\partial_i)$ with quadratic differentials~$\partial_i$. Let us equip the algebra
 \[
\mathcal{B}:=\bigotimes_{i\in Q^{(\infty)}_0} \mathcal{B}_i
 \]  
with the differential given by $\partial:=\partial^{(\infty)}+\sum_{i\in Q^{(\infty)}_0}\partial_i$. We claim that 
$H_\bullet(\mathcal{B},\partial)\cong\mathcal{A}_Q$.   
To establish this result, we note that if we first compute the homology of the differential $\sum_{i\in Q^{(\infty)}_0}\partial_i$, the K\"unneth formula implies that the result is precisely $\mathcal{A}_{Q^{(\infty)}}$, and then the second page of the spectral sequence of a bicomplex computes the homology of $\partial^{(\infty)}$, proving the required statement. It remains to notice that the differential $\partial$ is quadratic, which immediately implies that the algebra $\mathcal{A}_Q$ is Koszul.
\end{proof}

%\bibliographystyle{plain}
%\bibliography{biblio}

\printbibliography

\end{document}